\documentclass[12pt,american]{article}
\usepackage[T1]{fontenc}
\usepackage[latin9]{inputenc}
\usepackage{color}
\usepackage{babel}
\usepackage{units}
\usepackage{amsmath}
\usepackage{amsthm}
\usepackage{amssymb}
\usepackage[unicode=true,
 bookmarks=true,bookmarksnumbered=true,bookmarksopen=false,
 breaklinks=true,pdfborder={0 0 1},backref=false,colorlinks=true]
 {hyperref}
\hypersetup{pdftitle={Green Measures for Time Changed Markov Processes},
 pdfauthor={Jos{\'e} L. da Silva and Yuri Kondratiev},
 pdfsubject={Green Measures for Time Changed Markov Processes},
 pdfkeywords={Markov processes, Green measures, time change, asymptotic behavior},
 linkcolor=black,citecolor=black,urlcolor=black,filecolor=black}

\makeatletter
\usepackage{soul}
\theoremstyle{plain}
\newtheorem{thm}{\protect\theoremname}[section]
\theoremstyle{definition}
\newtheorem{example}[thm]{\protect\examplename}
\theoremstyle{remark}
\newtheorem{rem}[thm]{\protect\remarkname}
\theoremstyle{plain}
\newtheorem{lem}[thm]{\protect\lemmaname}
\theoremstyle{definition}
\newtheorem{defn}[thm]{\protect\definitionname}

\usepackage{textcomp}
\usepackage{mathrsfs}
\usepackage{bbm}

\usepackage{datetime}\usepackage{currfile}\usepackage{pdfsync}
\numberwithin{equation}{section}

\newcommand{\R}{{\mathbb R}}
\newcommand{\X}{{\R^d}}

\newcommand{\E}{{\mathbb E}}

\newcommand{\ga}{\gamma}

\providecommand{\definitionname}{Definition}

\makeatother

\providecommand{\definitionname}{Definition}
\providecommand{\examplename}{Example}
\providecommand{\lemmaname}{Lemma}
\providecommand{\remarkname}{Remark}
\providecommand{\theoremname}{Theorem}

\begin{document}
\title{Green Measures for Time Changed Markov Processes}
\author{\textbf{Yuri Kondratiev}\\
 Department of Mathematics, University of Bielefeld, \\
 D-33615 Bielefeld, Germany,\\
 Dragomanov University, Kiev, Ukraine\\
 Email: kondrat@math.uni-bielefeld.de\and\textbf{ Jos{\'e} Lu{\'i}s
da Silva},\\
 CIMA, University of Madeira, Campus da Penteada,\\
 9020-105 Funchal, Portugal.\\
 Email: joses@staff.uma.pt}
\date{\today}

\maketitle
 
\begin{abstract}
In this paper we study Green measures for certain classes of random
time change Markov processes where the random time change are inverse
subordinators. We show the existence of the Green measure for these
processes under the condition of the existence of the Green measure
of the original Markov processes and they coincide. Applications to
fractional dynamics in given.


\end{abstract}
\tableofcontents{}

\section{Introduction}

One of the most important questions in the theory of random processes
is related with the study of their asymptotic behavior. There are
several possibilities to formulate such questions. For example, let
$X(t)$, $t\ge0$ be a random process in $\X$ such that $X(0)=x\in\X$.
Denote by $\mu_{t}^{x}$ the one dimensional distribution of the process
at time $t$. Then the natural question is the limiting behavior of
$\mu_{t}^{x}$ for $t\to\infty$. Of course, we can expect a positive
answer to this question only for certain particular classes of these
processes. In the case of Markov processes, an essential technique
to study the time asymptotic is related with the Fokker-Planck equation
\[
\frac{\partial}{\partial t}\mu_{t}^{x}=L^{\ast}\mu_{t}^{x},
\]
where $L$ is the generator of the Markov process. In that case, $\mu_{t}^{x}$
is nothing but the transition probability measure $P_{t}(x,\mathrm{d}y)$
or the heat kernel for $L$ which may be analyzed for certain particular
cases, see e.g., \cite{Grigoryan2012,Grigoryan2018a} and the wide
list of references therein. On the other hand, even for very simple
classes of processes the time-space behavior of $P_{t}(x,\mathrm{d}y)$
may be very complicated, see \cite{Grigoryan2018} for the analysis
of continuous time random walks (compound Poisson processes) in $\X$.

An alternative way is to consider averaged characteristics of Markov
processes. In particular, we introduce the Green measure 
\[
\mathcal{G}(x,\mathrm{d}y):=\int_{0}^{\infty}P_{t}(x,\mathrm{d}y)\,\mathrm{d}t.
\]
The notion of the Green measure is closely related to the concept
of potential in stochastic analysis, see \cite{KdS2020} for details.
In the latter paper we have shown the existence of Green measures
for certain classes o Markov processes and analyzed their properties.

In this paper we are interested in transformations of Markov processes
by means of independent random time changes. The resulting process
is again a Markov process. In particular, as random time change we
consider inverse of subordinators. In the literature most of the results
in this direction are related with inverse stable subordinators, see
e.g., \cite{Meerschaert2002,Becker-Kern2004,Meerschaert2006,Baeumer2009,Magdziarz2015}.
In \cite{Kobayashi2020} the authors study the spectral heat contents
for time changed Brownian motions where the time change is given either
by a subordinator or an inverse subordinator with the underlying Laplace
exponent being regularly varying at infinity with index $\beta\in(0,1)$.
But it is also possible to consider more general inverse subordinators
and study such random processes and related properties, see e.g.,
\cite{dSKK16,KKdS19,KKS2018}.

Let $X$ be a Markov process which admits a Green measure $\mathcal{G}(x,\mathrm{d}y)$
and $Y$ its random time change by an inverse subordinator. Our aim
is to study the asymptotic behaviour of the process $Y$ for different
classes of inverse subordinators. The first main point concerns the
existence of Green measures for these time changed processes, that
is, applying the above definition of Green measure leads to divergent
integrals in all interesting cases. To overcome this difficulty we
introduce the concept of renormalized Green measure 
\[
\mathcal{G}_{r}(x,\mathrm{d}y):=\lim_{T\to\infty}\frac{1}{N(T)}\int_{0}^{T}\nu_{t}^{x}(\mathrm{d}y)\,\mathrm{d}t,
\]
where $\nu_{t}^{x}$ in the marginal distribution of $Y(t)$. The
renormalization $N(T)$ is uniquely defined by the inverse subordinator
under consideration, see \eqref{eq:limit-G-k} and \eqref{eq:normalization}
below. Such kind of normalizations are well known in the theory of
additive functionals for random processes. This enable us to state
the main contribution of this paper as follows: If the initial Markov
process has a Green measure, then the time changed process will have
a renormalized Green measure which coincides with the Green measure
for the Markov process, see Theorem \ref{thm:main-RTC-GM} below.
An interpretation of this result is very easy. In the time changed
process the evolution is delayed by the random environment and as
a result is slower. That means a slower decay in $t$ leads to a divergent
integral in the definition of the Green measure.

The paper is organized as follows. In Section \ref{sec:RT-FA} we
describe the class of subordinators we are interested in as well as
the corresponding inverse subordinators. The main assumption of these
classes is given in terms of the corresponding Laplace exponent, see
assumption (H) below. In addition, we recall the a result on the asymptotic
relating the density of the inverse subordinator and admissible kernels
satisfying (H), see Theorem \ref{thm:main-result}. We provide many
examples which fulfill the assumptions in Example \ref{exa:alpha-stable1}.
Sections \ref{sec:MP-RT} and \ref{sec:RGM} we introduce the main
object needed and show the main result of the paper, see Theorem \ref{thm:main-RTC-GM}.
Finally, in Section \ref{sec:Applications-FD} we make an application
to fractional dynamics for the special class of Markov processes known
as compound Poisson processes. More precisely, if $u(t,x)$ denotes
the solution of the Kolmogorov equation and $v(t,x)$ is the solution
of the associated fractional evolution equation, then the following
average result holds, see Theorem \ref{thm:asymptotic-FKE}
\[
\frac{1}{N(t)}\int_{0}^{t}v(s,x)\,\mathrm{d}s\sim\int_{0}^{\infty}u(s,x)\,\mathrm{d}s=\int_{\X}f(y)\mathcal{G}(x,\mathrm{d}y),\quad t\to\infty,
\]
where $f$ is a suitable initial data. 

\section{Random Times and Fractional Analysis}

\label{sec:RT-FA}In this section we introduce the classes of inverse
subordinators we are interested in. Associated to these classes we
define a kernel $k\in L_{\mathrm{loc}}^{1}(\mathbb{R}_{+})$ which
is used to define a general fractional derivatives (GFD), see \cite{Kochubei11}
for details and applications to fractional differential equations.
These admissible kernels $k$ are characterized in terms of their
Laplace transforms $\mathcal{K}(\lambda)$ as $\lambda\to0$, see
assumption (H) below.

Let $S=\{S(t),\;t\ge0\}$ be a subordinator without drift starting
at zero, that is, an increasing L{\'e}vy process starting at zero,
see \cite{Bertoin96} for more details. The Laplace transform of $S(t)$,
$t\ge0$ is expressed in terms of a Bernstein function $\Phi:[0,\infty)\longrightarrow[0,\infty)$
(also known as Laplace exponent) by 
\[
\mathbb{E}(e^{-\lambda S(t)})=e^{-t\Phi(\lambda)},\quad\lambda\ge0.
\]
The function $\Phi$ admits the representation 
\begin{equation}
\Phi(\lambda)=\int_{(0,\infty)}(1-e^{-\lambda\tau})\,\mathrm{d}\sigma(\tau),\label{eq:Levy-Khintchine}
\end{equation}
where the measure $\sigma$ (called L{\'e}vy measure) has support
in $[0,\infty)$ and fulfills 
\begin{equation}
\int_{(0,\infty)}(1\wedge\tau)\,\mathrm{d}\sigma(\tau)<\infty.\label{eq:Levy-condition}
\end{equation}
In what follows we assume that the L{\'e}vy measure $\sigma$ satisfy
\begin{equation}
\sigma\big((0,\infty)\big)=\infty.\label{eq:Levy-massumption}
\end{equation}
Using the L{\'e}vy measure $\sigma$ we define the kernel $k$ as
follows 
\begin{equation}
k:(0,\infty)\longrightarrow(0,\infty),\;t\mapsto k(t):=\sigma\big((t,\infty)\big).\label{eq:k}
\end{equation}
Its Laplace transform is denoted by $\mathcal{K}$, that is, for any
$\lambda\ge0$ one has 
\begin{equation}
\mathcal{K}(\lambda):=\int_{0}^{\infty}e^{-\lambda t}k(t)\,\mathrm{d}t.\label{eq:LT-k}
\end{equation}
The relation between the function $\mathcal{K}$ and the Laplace exponent
$\Phi$ is given by 
\begin{equation}
\Phi(\lambda)=\lambda\mathcal{K}(\lambda),\quad\forall\lambda\ge0.\label{eq:Laplace-exponent}
\end{equation}

In what follows we make the following assumption on the Laplace exponent
$\Phi(\lambda)$ of the subordinator $S$. 
\begin{description}
\item [{(H)}] $\Phi$ is a complete Bernstein function (that is, the L{\'e}vy
measure $\sigma$ is absolutely continuous with respect to the Lebesgue
measure) and the functions $\mathcal{K}$, $\Phi$ satisfy 
\begin{equation}
\mathcal{K}(\lambda)\to\infty,\text{ as \ensuremath{\lambda\to0}};\quad\mathcal{K}(\lambda)\to0,\text{ as \ensuremath{\lambda\to\infty}};\label{eq:H1}
\end{equation}
\begin{equation}
\Phi(\lambda)\to0,\text{ as \ensuremath{\lambda\to0}};\quad\Phi(\lambda)\to\infty,\text{ as \ensuremath{\lambda\to\infty}}.\label{eq:H2}
\end{equation}
\end{description}
\begin{example}
\label{exa:alpha-stable1} 
\begin{enumerate}
\item A classical example of a subordinator $S$ is the so-called $\alpha$-stable
process with index $\alpha\in(0,1)$. Specifically, a subordinator
is $\alpha$-stable if its Laplace exponent is 
\[
\Phi(\lambda)=\lambda^{\alpha}=\frac{\alpha}{\Gamma(1-\alpha)}\int_{0}^{\infty}(1-e^{-\lambda\tau})\tau^{-1-\alpha}\,\mathrm{d}\tau.
\]
In this case it follows that the L{\'e}vy measure is $\mathrm{d}\sigma_{\alpha}(\tau)=\frac{\alpha}{\Gamma(1-\alpha)}\tau^{-(1+\alpha)}\,\mathrm{d}\tau$,
the corresponding kernel $k_{\alpha}$ has the form $k_{\alpha}(t)=g_{1-\alpha}(t):=\frac{t^{-\alpha}}{\Gamma(1-\alpha)}$,
$t\ge0$ and its Laplace transform is $\mathcal{K}_{\alpha}(\lambda)=\lambda^{\alpha-1}$,
$\lambda\ge0$. 
\item The Gamma process $Y^{(a,b)}$ with parameters $a,b>0$ is another
example of a subordinator with Laplace exponent 
\[
\Phi_{(a,b)}(\lambda)=a\log\left(1+\frac{\lambda}{b}\right)=\int_{0}^{\infty}(1-e^{-\lambda\tau})a\tau^{-1}e^{-b\tau}\,\mathrm{d}\tau,
\]
the second equality is known as the Frullani integral. The L{\'e}vy
measure is given by $d\sigma_{(a,b)}(\tau)=a\tau^{-1}e^{-b\tau}\,\mathrm{d}\tau.$
The associated kernel $k_{(a,b)}(t)=a\Gamma(0,bt)$, $t>0$ and its
Laplace transform is $\mathcal{K}_{(a,b)}(\lambda)=a\lambda^{-1}\log(1+\frac{\lambda}{b})$,
$\lambda>0$. 
\item The truncated $\alpha$-stable subordinator (see Example 2.1-(ii)
in \cite{Chen2017}) $S_{\delta}$, $\delta>0$ is a driftless $\alpha$-stable
subordinator with L{\'e}vy measure given by 
\[
\mathrm{d}\sigma_{\delta}(\tau):=\frac{\alpha}{\Gamma(1-\alpha)}\tau^{-(1+\alpha)}1\!\!1_{(0,\delta]}(\tau)\,\mathrm{d}\tau,\qquad\delta>0.
\]
The corresponding Laplace exponent is 
\[
\Phi_{\delta}(\lambda)=\lambda^{\alpha}\left(1-\frac{\Gamma(-\alpha,\delta\lambda)}{\Gamma(-\alpha)}\right)+\frac{\delta^{-\alpha}}{\Gamma(1-\alpha)},
\]
where $\Gamma(\nu,z):=\int_{z}^{\infty}e^{-t}t^{\nu-1}\,\mathrm{d}t$
is the incomplete gamma function (see Section 8.3 in \cite{GR81})
and the associated kernel $k_{\delta}$ is given by 
\[
k_{\delta}(t):=\sigma_{\delta}\big((t,\infty)\big)=\frac{1\!\!1_{(0,\delta]}(t)}{\Gamma(1-\beta)}(t^{-\beta}-\delta^{-\beta}),\;t>0.
\]
\item Let $0<\beta<1$ and $0<\alpha<1$ be given and $S_{\alpha,\beta}(t)$,
$t\ge0$ the driftless subordinator with Laplace exponent given by
\[
\Phi_{\alpha,\beta}(\lambda)=\lambda^{\alpha}+\lambda^{\beta}.
\]
It is clear from item 1 above that the corresponding L{\'e}vy measure
$\sigma_{\alpha,\beta}$ is the sum of two L{\'e}vy measures, that
is, 
\[
\mathrm{d}\sigma_{\alpha,\beta}(\tau)=\mathrm{d}\sigma_{\alpha}(\tau)+\mathrm{d}\sigma_{\alpha}(\tau)=\frac{\alpha}{\Gamma(1-\alpha)}\tau^{-(1+\alpha)}\,\mathrm{d}\tau+\frac{\beta}{\Gamma(1-\beta)}\tau^{-(1+\beta)}\,\mathrm{d}\tau.
\]
Then the associated kernel $k_{\alpha,\beta}$ is 
\[
k_{\alpha,\beta}(t):=g_{1-\alpha}(t)+g_{1-\beta}(t)=\frac{t^{-\alpha}}{\Gamma(1-\alpha)}+\frac{t^{-\beta}}{\Gamma(1-\beta)},\;t>0
\]
and its Laplace transform is $\mathcal{K}_{\alpha,\beta}(\lambda)=\mathcal{K}_{\alpha}(\lambda)+\mathcal{K}_{\beta}(\lambda)=\lambda^{\alpha-1}+\lambda^{\beta-1}$,
$\lambda>0$. 
\item Kernel with exponential weight. Given $\gamma>0$ and $0<\alpha<1$
consider the subordinator with Laplace exponent 
\[
\Phi_{\gamma}(\lambda):=(\lambda+\gamma)^{\alpha}=\left(\frac{\lambda+\gamma}{\lambda}\right)^{1+\alpha}\frac{\alpha}{\Gamma(1-\alpha)}\int_{0}^{\infty}(1-e^{-\lambda\tau})\tau^{-1-\alpha}\,\mathrm{d}\tau.
\]
It follows that the L{\'e}vy measure is given by $\mathrm{d}\sigma_{\gamma}(\tau)=\left(\frac{\lambda+\gamma}{\lambda}\right)^{1+\alpha}\frac{\alpha}{\Gamma(1-\alpha)}\tau^{-(1+\alpha)}\mathrm{d}\tau$
which yields a kernel $k_{\gamma}$ with exponential weight, namely
\[
k_{\gamma}(t)=g_{1-\alpha}(t)e^{-\gamma t}=\frac{t^{-\alpha}}{\Gamma(1-\alpha)}e^{-\gamma t}.
\]
The corresponding Laplace transform of $k_{\gamma}$ is given by $\mathcal{K}_{\gamma}(\lambda)=\lambda^{-1}(\lambda+\gamma)^{\alpha}$,
$\lambda>0$. 
\end{enumerate}
\end{example}

Denote by $E$ the inverse process of the subordinator $S$, that
is, 
\begin{equation}
E(t):=\inf\{s\ge0\mid S(s)\ge t\}=\sup\{s\ge0\mid S(t)\le s\}.\label{eq:inverse-sub}
\end{equation}
For any $t\ge0$ we denote by $G_{t}^{k}(\tau):=G_{t}(\tau)$, $\tau\ge0$
the marginal density of $E(t)$ or, equivalently 
\[
G_{t}(\tau)\,\mathrm{d}\tau=\partial_{\tau}P(E(t)\le\tau)=\partial_{\tau}P(S(\tau)\ge t)=-\partial_{\tau}P(S(\tau)<t).
\]

As the density $G_{t}(\tau)$ plays an important role in what follows,
we collect the most important properties of it. 
\begin{rem}
\label{rem:distr-alphastab-E}If $S$ is the $\alpha$-stable process,
$\alpha\in(0,1)$, then the inverse process $E(t)$ has a Mittag-Leffler
distribution (cf.~Prop.~1(a) in \cite{Bingham1971}), namely 
\begin{equation}
\mathbb{E}(e^{-\lambda E(t)})=\int_{0}^{\infty}e^{-t\tau}G_{t}(\tau)\,\mathrm{d}\tau=\sum_{n=0}^{\infty}\frac{(-\lambda t^{\alpha})^{n}}{\Gamma(n\alpha+1)}=E_{\alpha}(-\lambda t^{\alpha}).\label{eq:Laplace-density-alpha}
\end{equation}
It follows from the asymptotic behavior of the Mittag-Leffler function
$E_{\alpha}$ that $\mathbb{E}(e^{-\lambda E(t)})\sim Ct^{-\alpha}$
as $t\to\infty$. Using the properties of the Mittag-Leffler function
$E_{\alpha}$, we can show that the density $G_{t}(\tau)$ is given
in terms of the Wright function $W_{\mu,\nu}$, namely $G_{t}(\tau)=t^{-\alpha}W_{-\alpha,1-\alpha}(\tau t^{-\alpha})$,
see \cite{Gorenflo1999} for more details. 
\end{rem}

For a general subordinator, the following lemma determines the $t$-Laplace
transform of $G_{t}(\tau)$, with $k$ and $\mathcal{K}$ given in
\eqref{eq:k} and \eqref{eq:LT-k}, respectively. For the proof see
\cite{Kochubei11} or Lemma~3.1 in \cite{Toaldo2015}. 
\begin{lem}
\label{lem:t-LT-G}The $t$-Laplace transform of the density $G_{t}(\tau)$
is given by 
\begin{equation}
\int_{0}^{\infty}e^{-\lambda t}G_{t}(\tau)\,\mathrm{d}t=\mathcal{K}(\lambda)e^{-\tau\lambda\mathcal{K}(\lambda)}.\label{eq:LT-G-t}
\end{equation}
The double ($\tau,t$)-Laplace transform of $G_{t}(\tau)$ is 
\begin{equation}
\int_{0}^{\infty}\int_{0}^{\infty}e^{-p\tau}e^{-\lambda t}G_{t}(\tau)\,\mathrm{d}t\,\mathrm{d}\tau=\frac{\mathcal{K}(\lambda)}{\lambda\mathcal{K}(\lambda)+p}.\label{eq:double-Laplace}
\end{equation}
\end{lem}

For any $\alpha\in(0,1)$ the Caputo-Dzhrbashyan fractional derivative
of order $\alpha$ of a function $u$ is defined by (see e.g., \cite{KST2006}
and references therein) 
\begin{equation}
\big(\mathbb{D}_{t}^{\alpha}u\big)(t)=\frac{d}{dt}\int_{0}^{t}k_{\alpha}(t-\tau)u(\tau)\,\mathrm{d}\tau-k_{\alpha}(t)u(0),\quad t>0,\label{eq:Caputo-derivative}
\end{equation}
where $k_{\alpha}$ is given in Example \ref{exa:alpha-stable1}-(1),
that is, $k_{\alpha}(t)=g_{1-\alpha}(t)=\frac{t^{-\alpha}}{\Gamma(1-\alpha)}$,
$t>0$. In general, starting with a subordinator $S$ and the kernel
$k\in L_{\mathrm{loc}}^{1}(\mathbb{R}_{+})$ given as in \eqref{eq:k},
we may define a differential-convolution operator by 
\begin{equation}
\big(\mathbb{D}_{t}^{(k)}u\big)(t)=\frac{d}{dt}\int_{0}^{t}k(t-\tau)u(\tau)\,\mathrm{d}\tau-k(t)u(0),\;t>0.\label{eq:general-derivative}
\end{equation}
The operator $\mathbb{D}_{t}^{(k)}$ is also known as generalized
fractional derivative. The distributed order derivative $\mathbb{D}_{t}^{(\mu)}$
is an example of such operator, corresponding to 
\begin{equation}
k(t)=\int_{0}^{1}g_{1-\alpha}(t)\,\mathrm{d}\alpha=\int_{0}^{1}\frac{t^{-\alpha}}{\Gamma(1-\alpha)}\mu(\alpha)\,\mathrm{d}\alpha,\quad t>0,\label{eq:distributed-kernel}
\end{equation}
where $\mu(\tau)$, $0\le\tau\le1$ is a positive weight function
on $[0,1]$, see \cite{Atanackovic2009,Daftardar-Gejji2008,Hanyga2007,Kochubei2008,Gorenflo2005,Meerschaert2006}
for applications.

Now we introduce a suitable class of admissible $k(t)$ and state
and essential theorem which this class obeys, see Theorem \ref{thm:main-result}
below. 
\begin{defn}[Admissible kernels - $\mathbb{K}(\mathbb{R}_{+})$]
The subset $\mathbb{K}(\mathbb{R}_{+})\subset L_{\mathrm{loc}}^{1}(\mathbb{R}_{+})$
of admissible kernels $k$ is defined by those elements in $L_{\mathrm{loc}}^{1}(\mathbb{R}_{+})$
satisfying (H) such that for some $s_{0}>0$ 
\begin{equation}
\liminf_{\lambda\to0+}\frac{1}{\mathcal{K}(\lambda)}\int_{0}^{\nicefrac{s_{0}}{\lambda}}k(t)\,dt>0\tag*{(A1)}\label{eq:A1}
\end{equation}
and 
\begin{equation}
\lim_{{t,r\to\infty\atop \frac{t}{r}\to1}}\left(\int_{0}^{t}k(s)\,ds\right)\left(\int_{0}^{r}k(s)\,ds\right)^{-1}=1.\tag*{(A2)}\label{eq:A2}
\end{equation}
\end{defn}

The following theorem establishes an asymptotic relation between the
density $G_{t}(\tau)$ and the kernel $k\in\mathbb{K}(\mathbb{R}_{+})$.
For the proof, see \cite{KKdS19}. 
\begin{thm}
\label{thm:main-result}Let $\tau\in[0,\infty)$ be fixed and $k\in\mathbb{K}(\mathbb{R}_{+})$
a given admissible kernel. Define the map $G_{\cdot}(\tau):[0,\infty)\longrightarrow\mathbb{R}_{+}$,
$t\mapsto G_{t}(\tau)$ such that $\int_{0}^{\infty}e^{-\lambda t}G_{t}(\tau)\,dt$
exists for all $\lambda>0$. Then 
\begin{equation}
\lim_{t\to\infty}\left(\int_{0}^{t}G_{s}(\tau)\,ds\right)\left(\int_{0}^{t}k(s)\,ds\right)^{-1}=1\label{eq:limit-G-k}
\end{equation}
or 
\[
M_{t}\big(G_{t}(\tau)\big):=\frac{1}{t}\int_{0}^{t}G_{s}(\tau)\,ds\sim\frac{1}{t}\int_{0}^{t}k(s)\,ds=:M_{t}\big(k(t)\big),\quad t\to\infty
\]
and $M_{t}\big(G_{t}(\tau)\big)$ is uniformly bounded in $\tau\in\mathbb{R}_{+}$. 
\end{thm}

\section{Markov Processes in Random Time}

\label{sec:MP-RT}Let $X=\{X(t),\;t\ge0\}$ be a Markov process in
$\X$ such that $X(0)=x\in\X$ almost surely. We are interested in
a new process $Y=\{Y(t),\;t\ge0\}$ which is constructed by a random
time change in $X$. Namely, if $E(t)$, $t\ge0$ denotes (as in Section
\ref{sec:RT-FA}) the inverse of a subordinator $S$ independent of
$X$, then we define $Y$ by 
\[
Y(t):=X(E(t)),\quad t\ge0.
\]
Note that inverse subordinators have found many applications in probability
theory, see \cite{Magdziarz2015} for a detailed discussions and several
related references. In particular, for their relationship with local
times of some Markov processes, see \cite{Bertoin96}. Similarities
between inverse subordinators and renewal processes also are well
studied. There are important applications of inverse subordinators
in finance and physics. We stress that random time processes may be
considered as mathematical realizations of the general concept of
biological time known in biology and ecology since the pioneering
works of V.~I.~Vernadsky \cite{Vernadsky1998}.

The first natural question which appear here concerns the possible
relations between the characteristics of the processes $X(t)$ and
$Y(t)$. To the best of our knowledge this question was for the first
time discussed by A.~Mura, M.S.~Taqqu and F.~Mainardi in \cite{Mura_Taqqu_Mainardi_08}.
The authors considered the diffusion processes with an implicitly
defined class of random times $E(t)$. Later similar questions were
discussed by several authors, see e.g., \cite{Toaldo2015} and references
therein.

The situation there may be described as follows. Define a function
\[
u(t,x):=\E[f(X(t))],\;t>0,\;x\in\X
\]
for a proper $f:\X\to\R$. This is the solution of the Kolmogorov
equation 
\begin{equation}
\frac{\partial}{\partial t}u(t,x)=Lu(t,x),\label{KE}
\end{equation}

\[
u(0,x)=f(x),
\]
where $L$ is the generator of the process $X(t)$. Let us define
a similar function for $Y(t)$: 
\[
v(t,x)=\E[f(Y(t))].
\]
Then this function satisfies the following fractional evolution equation:
\begin{equation}
D_{t}^{(k)}v(t,x)=Lv(t,x).\label{FKE}
\end{equation}
Moreover, the subordination formula holds:

\begin{equation}
v(t,x)=\int_{0}^{\infty}u(\tau,x)G_{t}(\tau)\,\mathrm{d}\tau,\label{SUB}
\end{equation}
where, as before, $G_{t}(\tau)$ is the density of the inverse subordinator
$E(t)$.

If $\mu_{t}^{x}$ and $\nu_{t}^{x}$ denote the marginal distributions
of $X(t)$ and $Y(t)$, respectively, then the subordination relations
for these distributions is given by 
\begin{equation}
\nu_{t}^{x}=\int_{0}^{\infty}\mu_{\tau}^{x}G_{t}(\tau)\,\mathrm{d}\tau.\label{SUBM}
\end{equation}
In the next section we use these relations to study the renormalized
Green measure associated to the subordinated process $Y$.

\section{Renormalized Green Measures}

\label{sec:RGM}Let $X$ be a Markov process and $Y$ be the time
changed process as in Section \ref{sec:MP-RT} with all our notations
from there. For every jump of the subordinator $S$ there is a corresponding
flat period of its inverse $E$. These flat periods represent trapping
events in which the test particle gets immobilized in a trap. Trapping
slows down the overall dynamics of the initial Markov process $X$.
Our aim is to analyze how these traps will be reflected in the asymptotic
behavior of the changed process $Y$.

To study the time asymptotic of random processes there is the useful
notion of Green measures, see for example \cite{KdS2020} for this
notion. More precisely, if $Z(t)$, $t\ge0$ is a random process in
$\X$ with $Z(0)=x\in\X$ and, for each $t\ge0$, $\ga_{t}^{x}$ denotes
its marginal distribution, then the Green measure of $Z$ is defined
by 
\[
\mathcal{G}(x,\mathrm{d}y):=\int_{0}^{\infty}\ga_{t}^{x}(\mathrm{d}y)\,\mathrm{d}t
\]
if this integral converges. In \cite{KdS2020} we have shown the existence
of the Green measures for certain classes of Markov processes in $\X$
with the necessary condition $d\geq3$. For $d=1,2$ we have to modify
this definition by means of a renormalized Green measure, namely 
\[
\mathcal{G}_{r}(x,\mathrm{d}y)=\lim_{T\to\infty}\frac{1}{N(T)}\int_{0}^{T}\mu_{t}^{x}(\mathrm{d}y)\,\mathrm{d}t.
\]
This approach (in a bit different framework) is well know in the theory
of additive functionals for Markov processes, see \cite{KMS20} for
an extended list of references.

The following lemma shows that the Green measure for $Y(t)$ does
not exists for a general inverse subordinator and arbitrary Markov
process $X(t)$. 
\begin{lem}
\label{lem:Green-meas-lemma}Under the assumptions formulated above
for any dimension $d$ the Green measure for $Y(t)$ does not exists. 
\end{lem}

\begin{proof}
Using the subordination formula (\ref{SUBM}) we obtain 
\[
\int_{0}^{\infty}\nu_{t}^{x}\,\mathrm{d}t=\int_{0}^{\infty}\int_{0}^{\infty}\mu_{\tau}^{x}G_{t}(\tau)\,\mathrm{d}\tau\,\mathrm{d}t.
\]
But we know that for each $\tau$, it follows from \eqref{eq:H1},
\eqref{eq:H2} and \eqref{eq:LT-G-t} that 
\[
\int_{0}^{\infty}G_{t}(\tau)\,\mathrm{d}t=\mathcal{K}(0)=+\infty.
\]
Therefore, the considered integral is divergent. 
\end{proof}
As the Green measure does not exists for a general subordinated process
$Y$, we have to consider instead a renormalized Green measure. More
precisely, we would like to find the following limit 
\[
\mathcal{G}_{r}(x,\mathrm{d}y):=\lim_{T\to\infty}\frac{1}{N(T)}\int_{0}^{T}\nu_{t}^{x}(\mathrm{d}y)\,\mathrm{d}t.
\]

\begin{thm}
\label{thm:main-RTC-GM}Assume that the Markov process $X(t)$ in
$\X$, $d\geq3$ has a Green measure $\mathcal{G}(x,\mathrm{d}y)$
and define 
\begin{equation}
N(T):=\int_{0}^{T}k(s)\,\mathrm{d}s,\quad T\ge0.\label{eq:normalization}
\end{equation}
Then the renormalized Green measure for $Y(t)$ exists and 
\[
\mathcal{G}_{r}(x,\mathrm{d}y)=\mathcal{G}(x,\mathrm{d}y).
\]
\end{thm}

\begin{proof}
Using the subordination relation \eqref{SUBM} the renormalized Green
measure $\mathcal{G}_{r}(x,\mathrm{d}y)$ may be written as 
\[
\mathcal{G}_{r}(x,\mathrm{d}y)=\lim_{T\to\infty}\frac{1}{N(T)}\int_{0}^{T}\int_{0}^{\infty}\mu_{\tau}^{x}(\mathrm{d}y)G_{t}(\tau)\,\mathrm{d}\tau\,\mathrm{d}t.
\]
Now using Fubini theorem and Theorem \ref{thm:main-result} it follows
that 
\[
\mathcal{G}_{r}(x,\mathrm{d}y):=\lim_{T\to\infty}\frac{1}{N(T)}\int_{0}^{T}\nu_{t}^{x}(\mathrm{d}y)\,\mathrm{d}t=\int_{0}^{\infty}\mu_{t}^{x}(\mathrm{d}y)\,\mathrm{d}t=\mathcal{G}(x,\mathrm{d}y).
\]
This shows the statement of the theorem and finish the proof. 
\end{proof}
\begin{rem}
As we mentioned at the beginning of this section, random time produces
trapping (or environments, or friction) effects in the Markov dynamics.
That is one reason why in physics such processes are very useful.
As the trapping slows down the Markov dynamics, then the usual definition
of Green measures produces a divergent integral. To compensate this
divergence we have to consider a renormalization with a time depended
factor. The time asymptotic of the renormalized Green measure coincides
with the Green measure of the initial Markov process. 
\end{rem}

\section{Applications to Fractional Dynamics}

\label{sec:Applications-FD}Let $u(t,x)$ be the solution of equation
\eqref{KE} and $v(t,x)$ the corresponding solution of the fractional
equation \eqref{FKE}. Our goal is to compare the behaviors in $t$
for these solutions. To this end, at first we restrict the class of
Markov processes under considerations. Namely, let $a:\X\to\R$ be
a fixed kernel with the following properties: 
\begin{enumerate}
\item Symmetric, $a(-x)=a(x)$, for every $x\in\X$. 
\item Positive, continuous and bounded, $a\geq0$, $a\in C_{b}(\X)$. 
\item Integrable 
\[
\int_{\X}a(y)\,\mathrm{d}y=1.
\]
\end{enumerate}
Consider the generator $L$ defined by 
\[
(Lf)(x)=\int_{\X}a(x-y)[f(y)-f(x)]\,\mathrm{d}y=(a*f)(x)-f(x),\quad x\in\X.
\]
In particular, $L^{\ast}=L$ in $L^{2}(\X)$ and $L$ is a bounded
linear operator in all $L^{p}(\X)$, $p\ge1$. We call this operator
the jump generator with jump kernel $a$. The corresponding Markov
process is of a pure jump type and is known in stochastic as compound
Poisson process, see \cite{Skorohod1991}.

We make the following assumptions on the kernel $a$. 
\begin{description}
\item [{(A)}] The jump kernel $a$ is such that the Fourier transform $\hat{a}\in L^{1}(\X)$
and it has finite second moment, that is, 
\[
\int_{\X}|x|^{2}a(x)\,\mathrm{d}x<\infty.
\]
\end{description}
Define the Banach space $CL(\X)$ as the set of all bounded continuous
and integrable functions on $\X$, that is, $CL(\X)=C_{b}(\X)\cap L^{1}(\X)$.
The norm in this space is constructed as the sum of $C_{b}(\X)$ and
$L^{1}(\X)$ norms. The following theorem was shown in \cite{KdS2020}. 
\begin{thm}
\label{thm:average}Let $a$ be a kernel which satisfies all the above
assumptions and $d\geq3$. Consider the solution $u(t,x)$ to \eqref{KE}
with an initial data $f\in CL(\X)$. Then 
\[
\int_{0}^{\infty}u(t,x)\,\mathrm{d}t=\int_{\X}f(y)\mathcal{G}(x,\mathrm{d}y),
\]

where $\mathcal{G}(x,\mathrm{d}y)$ is the Green measure of the corresponding
Markov jump process. 
\end{thm}

This result gives an averaged characteristic of the dynamics $u(t,x)$
corresponding to the Markov processes via the Kolmogorov equation.
On the other hand, for the fractional dynamics $v(t,x)$ (the solution
of equation \eqref{FKE}) we have only information about the Cesaro
mean 
\[
M_{t}(f)=\frac{1}{t}\int_{0}^{t}v(s,x)\,\mathrm{d}s.
\]
The asymptotic of this mean was studied in \cite{KKdS19}. In particular,
when $\Phi(\lambda)=\lambda^{\alpha}$, $0<\alpha<1$ the kernel $k(t)=\frac{t^{-\alpha}}{\Gamma(1-\alpha)}$
and the GFD corresponds to the Caputo-Dzhrbashyan fractional derivative
of order $\alpha$. For this class of kernels we have 
\[
M_{t}(f)\sim Ct^{-\alpha},\quad t\to\infty.
\]

With the help of Theorem \ref{thm:average} we may also derive an
average result for the fractional dynamics $v(t,x)$. 
\begin{thm}
\label{thm:asymptotic-FKE}Under assumptions of Theorem \ref{thm:average}
holds

\[
\frac{1}{N(t)}\int_{0}^{t}v(s,x)\,\mathrm{d}s\sim\int_{\X}f(y)\mathcal{G}(x,\mathrm{d}y),\quad t\to\infty.
\]
\end{thm}

\begin{proof}
Using (\ref{SUB}) we have 
\[
\frac{1}{N(t)}\int_{0}^{t}v(s,x)\,\mathrm{d}s=\frac{1}{N(t)}\int_{0}^{t}\int_{0}^{\infty}u(\tau,x)G_{s}(\tau)\,\mathrm{d}\tau\,\mathrm{d}s.
\]
Again it follows from Fubini theorem, Theorem \ref{thm:main-result}
and the definition of $N(t)$ that 
\[
\frac{1}{N(t)}\int_{0}^{t}v(s,x)\,\mathrm{d}s=\int_{0}^{\infty}u(\tau,x)\left(\frac{1}{N(t)}\int_{0}^{t}G_{s}(\tau)\,\mathrm{d}s\right)\,\mathrm{d}\tau\underset{t\to\infty}{\longrightarrow}\int_{0}^{\infty}u(\tau,x)\,\mathrm{d}\tau.
\]
Then the result of the theorem follows from Theorem \ref{thm:average}. 
\end{proof}
\begin{rem}
\label{rem:space-decay} 
\begin{enumerate}
\item For concrete cases of jump kernels we have more information about
space decay of the Green measures, see \cite{KdS2020}. It gives the
possibility to extend the statement of Theorem \ref{thm:asymptotic-FKE}
to a wider class of the initial data $f$. 
\item The same result is true for the Brownian motion $B(t)$ in $\X$ for
$d\geq3$, see \cite{KdS2020}. 
\end{enumerate}
\end{rem}

\subsection*{Acknowledgments}

This work has been partially supported by Center for Research in Mathematics
and Applications (CIMA) related with the Statistics, Stochastic Processes
and Applications (SSPA) group, through the grant UIDB/MAT/04674/2020
of FCT-Funda{\c c\~a}o para a Ci{\^e}ncia e a Tecnologia, Portugal.


\begin{thebibliography}{KKdS20b}

\bibitem[APZ09]{Atanackovic2009}
T.~M. Atanackovic, S.~Pilipovic, and D.~Zorica.
\newblock {Time distributed-order diffusion-wave equation. I., II.}
\newblock In {\em Proceedings of the Royal Society of London A: Mathematical,
  Physical and Engineering Sciences}, volume 465, pages 1869--1891, 1893--1917.
  The Royal Society, 2009.

\bibitem[Ber96]{Bertoin96}
J.~Bertoin.
\newblock {\em L\'evy processes}, volume 121 of {\em Cambridge Tracts in
  Mathematics}.
\newblock Cambridge University Press, Cambridge, 1996.

\bibitem[Bin71]{Bingham1971}
N.~H. Bingham.
\newblock Limit theorems for occupation times of {M}arkov processes.
\newblock {\em Z.\ Wahrsch.\ verw.\ {G}ebiete}, 17:1--22, 1971.

\bibitem[BKMS04]{Becker-Kern2004}
P.~Becker-Kern, M.~M. Meerschaert, and H.-P. Scheffler.
\newblock Limit theorems for coupled continuous time random walks.
\newblock {\em The Annals of Probability}, 32(1B):730--756, 2004.

\bibitem[BMN09]{Baeumer2009}
B.~Baeumer, M.~Meerschaert, and E.~Nane.
\newblock {Brownian Subordinators and Fractional Cauchy Pproblems}.
\newblock {\em Trans.\ Amer.\ Math.\ Soc.}, 361(7):3915--3930, 2009.

\bibitem[Che17]{Chen2017}
Z.-Q. Chen.
\newblock Time fractional equations and probabilistic representation.
\newblock {\em Chaos Solitons Fractals}, 102:168--174, 2017.

\bibitem[DGB08]{Daftardar-Gejji2008}
V.~Daftardar-Gejji and S.~Bhalekar.
\newblock Boundary value problems for multi-term fractional differential
  equations.
\newblock {\em J. Math. Anal. Appl.}, 345(2):754--765, 2008.

\bibitem[dSKK16]{dSKK16}
J.~L. da~Silva, Y.~G. Kondratiev, and A.~N. Kochubei.
\newblock Fractional statistical dynamics and fractional kinetics.
\newblock {\em Methods Funct.\ Anal.\ Topology}, 22:197--209, 2016.

\bibitem[GHH18]{Grigoryan2018a}
A.~Grigor'yan, E.~Hu, and J.~Hu.
\newblock Two-sided estimates of heat kernels of jump type dirichlet forms.
\newblock {\em Adv.\ Math.}, 330:433--515, 2018.

\bibitem[GKPZ18]{Grigoryan2018}
A.~Grigor'yan, Yu.~G. Kondratiev, A.~Piatnitski, and E.~Zhizhina.
\newblock Pointwise estimates for heat kernels of convolution-type operators.
\newblock {\em Proc. Lond. Math. Soc. (3)}, 114(4):849--880, 2018.

\bibitem[GLM99]{Gorenflo1999}
R.~Gorenflo, Y.~Luchko, and F.~Mainardi.
\newblock {Analytical properties and applications of the Wright function}.
\newblock {\em Fract.\ Calc.\ Appl.\ Anal.}, 2(4):383--414, 1999.

\bibitem[GR15]{GR81}
I.~S. Gradstein and I.~M. Ryshik.
\newblock {\em Tables of Series, Products and Integrals}.
\newblock Academic Press, 225 Wyman Street, Waltham, MA 02451, USA, 8 edition,
  2015.

\bibitem[GT12]{Grigoryan2012}
A.~Grigor'yan and A.~Telcs.
\newblock Two-sided estimates of heat kernels on metric measure spaces.
\newblock {\em Ann. Probab.}, 40(3):1212--1284, 2012.

\bibitem[GU05]{Gorenflo2005}
R.~Gorenflo and S.~Umarov.
\newblock {Cauchy and nonlocal multi-point problems for distributed order
  pseudo-differential equations, Part one}.
\newblock {\em Z.\ Anal.\ Anwend.}, 24(3):449--466, 2005.

\bibitem[Han07]{Hanyga2007}
A.~Hanyga.
\newblock Anomalous diffusion without scale invariance.
\newblock {\em J.\ Phys.\ A: Mat.\ Theor.}, 40(21):5551, 2007.

\bibitem[KdS20]{KdS2020}
Yu.~G. Kondratiev and J.~L. da~Silva.
\newblock {Green Measures for Markov Processes}, 2020.
\newblock ArXiv:2006.07514. Submitted to Methods Funct.\ Anal.\ Topology.

\bibitem[KKdS20a]{KKS2018}
A.~Kochubei, Yu.~G. Kondratiev, and J.~L. da~Silva.
\newblock From random times to fractional kinetics.
\newblock {\em Interdisciplinary Studies of Complex Systems}, 16:5--32, 2020.

\bibitem[KKdS20b]{KKdS19}
A.~Kochubei, Yu.~G. Kondratiev, and J.~L. da~Silva.
\newblock Random time change and related evolution equations. {T}ime asymptotic
  behavior.
\newblock {\em Stochastics and Dynamics}, 4:2050034--1--24, 2020.

\bibitem[KMS20]{KMS20}
Y.~Kondratiev, Y.~Mishura, and G.~Shevchenko.
\newblock Limit theorems for additive functionals of continuous time random
  walks.
\newblock {\em Proceedings of the Royal Society of Edinburgh: Section A
  Mathematics}, pages 1--22, 2020.

\bibitem[Koc08]{Kochubei2008}
A.~N. Kochubei.
\newblock Distributed order calculus and equations of ultraslow diffusion.
\newblock {\em J.\ Math.\ Anal.\ Appl.}, 340(1):252--281, 2008.

\bibitem[Koc11]{Kochubei11}
A.~N. Kochubei.
\newblock General fractional calculus, evolution equations, and renewal
  processes.
\newblock {\em Integral Equations Operator Theory}, 71(4):583--600, October
  2011.

\bibitem[KP20]{Kobayashi2020}
K.~Kobayashi and H.~Park.
\newblock {Spectral heat content for time-changed killed Brownian motions},
  Arxiv:2007.05776v1, 2020.

\bibitem[KST06]{KST2006}
A.~A. Kilbas, H.~M. Srivastava, and J.~J. Trujillo.
\newblock {\em {Theory and Applications of Fractional Differential Equations}},
  volume 204 of {\em North-Holland Mathematics Studies}.
\newblock Elsevier Science B.V., Amsterdam, 2006.

\bibitem[MBSB02]{Meerschaert2002}
M.~M. Meerschaert, D.~A. Benson, H.-P. Scheffler, and B.~Baeumer.
\newblock Stochastic solution of space-time fractional diffusion equations.
\newblock {\em Phys. Rev. E}, 65(4):041103, 2002.

\bibitem[MS06]{Meerschaert2006}
M.~M. Meerschaert and H.-P. Scheffler.
\newblock Stochastic model for ultraslow diffusion.
\newblock {\em Stochastic Process.\ Appl.}, 116(9):1215--1235, 2006.

\bibitem[MS15]{Magdziarz2015}
M.~Magdziarz and R.~L. Schilling.
\newblock {Asymptotic properties of Brownian motion delayed by inverse
  subordinators}.
\newblock {\em Proceedings of the American Mathematical Society},
  143(10):4485--4501, 2015.

\bibitem[MTM08]{Mura_Taqqu_Mainardi_08}
A.~Mura, M.~S. Taqqu, and F.~Mainardi.
\newblock Non-{M}arkovian diffusion equations and processes: analysis and
  simulations.
\newblock {\em Phys.\ A}, 387(21):5033--5064, 2008.

\bibitem[Sko91]{Skorohod1991}
A.~V. Skorohod.
\newblock {\em {Random Processes with Independent Increments}}, volume~47 of
  {\em Mathematics and its applications (Soviet series)}.
\newblock Springer, 1991.

\bibitem[Toa15]{Toaldo2015}
B.~Toaldo.
\newblock Convolution-type derivatives, hitting-times of subordinators and
  time-changed {$C_0$}-semigroups.
\newblock {\em Potential Anal.}, 42(1):115--140, 2015.

\bibitem[Ver98]{Vernadsky1998}
V.~I. Vernadsky.
\newblock {\em The Biosphere: Complete Annotated Edition}.
\newblock Springer Science \& Business Media, 1998.

\end{thebibliography}

\end{document}